\documentclass{amsart}
\usepackage{standalone}
\usepackage{graphicx}
\usepackage{amssymb, amsfonts, amsmath, amsthm}
\usepackage[all]{xy}
\usepackage{amsmath}
\usepackage{verbatim}
\usepackage[switch, modulo]{lineno}
\usepackage{hyperref}
\usepackage{tikz-cd}
\newcommand{\abs}[1]{\left| #1 \right|}

\newcommand{\R}{\mathbb{R}}
\newcommand{\RP}{\mathbb{RP}}

\newcommand{\C}{\mathbb{C}}

\newcommand{\Q}{\mathbb{Q}}

\newcommand{\PGL}{\operatorname{PGL}}
\newcommand{\GL}{\operatorname{GL}}
\newcommand{\SL}{\operatorname{SL}}

\newcommand{\SO}{\operatorname{SO}}
\newcommand{\SU}{\operatorname{SU}}

\newcommand{\Z}{\mathbb{Z}}

\newcommand{\HH}{\mathbb{H}}

\newcommand{\bs}{\backslash}

\newcommand{\hull}{\mathcal{CH}}
\newcommand{\red}[1]{{\color{black} #1}}

\renewcommand{\O}{\mathcal{O}}

\newtheorem{theorem}{Theorem}

\newtheorem{proposition}[theorem]{Proposition}
\newtheorem{corollary}[theorem]{Corollary}
\newtheorem{lemma}[theorem]{Lemma}
\newtheorem{remark}[theorem]{Remark}

\numberwithin{equation}{section}
\numberwithin{theorem}{section}

\begin{document}

\title[Thin subgroups of $\SL(n+1,\R)$]{Constructing thin subgroups of $\SL(n+1,\R)$ via bending}
\author{Samuel Ballas and D.\ D.\  Long}
\email{ballas@math.fsu.edu}
\email{long@math.ucsb.edu}
\date{\today}
\address{Department of Mathematics\\ 
Florida State University\\ Tallahassee, FL 32306, USA}
\address{Department of Mathematics\\
University of California Santa Barbara\\
Santa Barbara, CA 93106, USA}

\maketitle

\begin{abstract}
	In this paper we use techniques from convex projective geometry to produce many new examples of thin subgroups of lattices in special linear groups that are isomorphic to the fundamental groups of finite volume hyperbolic manifolds. More specifically, we show that for a large class of arithmetic lattices in $\SO(n,1)$ it is possible to find infinitely many non-commensurable lattices in $\SL(n+1,\R)$ that contain a thin subgroup isomorphic to a finite index subgroup of the original arithmetic lattice. This class of arithmetic lattices includes all non-cocompact arithmetic lattices as well as all cocompact arithmetic lattices when $n$ is even.

	\end{abstract}

\tableofcontents

Let $G$ be a semi-simple Lie group and let $\Gamma\subset G$ be a lattice. A subgroup $\Delta\subset \Gamma$ is called \emph{a thin group} if $\Delta$ has infinite index in $\Gamma$ and is Zariski dense in $G$. Over the last several years, there has been a great deal of interest in thin subgroups of lattices in a variety of Lie groups \cite{FuchsMeiriSarnakThin,SarnakThinGroups,FuchsThin}. Much of this interest has been motivated by work of Bourgain, Gamburd, and Sarnak \cite{BourGambSarnAffineSeive} related to expanders and ``affine sieves.'' More generally, there is an increasingly strong sense that thin groups have many properties in common with lattices in $G$.

Furthermore, there is evidence that suggests that generic discrete subgroups of lattices are thin and free (see \cite{FuchsThin,FuchsRivinThin}). However, there is also great interest in constructing thin groups that are not free (or even decomposable as free products). For instance the seminal work of Kahn and Markovic \cite{KahnMark} constructs many thin subgroups contained in any cocompact lattice of $\SL(2,\C)$ that are isomorphic to the fundamental group of a closed surface. There are several generalizations of this result that exhibit thin surface groups in a variety of Lie groups. For instance, Cooper and Futer \cite{CooperFuterThin}, and independently Kahn and Wright \cite{KahnWright}, recently proved a similar result for non-compact lattices in $\SL(2,\C)$ and Kahn, Labourie and Mozes \cite{KahnLabourieMozes} proved an analogue for cocompact lattices in a large class of Lie groups.

These results naturally lead to the question of which isomorphism types of groups can occur as thin groups. In this paper we provide a partial answer by showing that in each dimension there are infinitely many finite volume hyperbolic manifolds whose fundamental groups arise as thin subgroups of lattices in special linear groups. Our main result is:

\begin{theorem}\label{mainthm}
	Let $\Gamma$ be a cocompact (resp.\ non-cocompact) arithmetic lattice in $\SO(n,1)$ of orthogonal type  then there are infinitely many non-commensurable cocompact (resp.\ non-cocompact) lattices in $\SL(n+1,\R)$ that each contain a thin subgroup isomorphic to a finite index subgroup of $\Gamma$. 
\end{theorem}

\red{The definition of an arithmetic lattice of orthogonal type is given in Section \ref{lattices_in_son1}.} It turns out that all non-cocompact arithmetic lattices in $\SO(n,1)$ are of orthogonal type (see the introduction of \cite{LiMillsonArithmetic} and \S 6.4 of \cite{WitteMorris}), and so we have the following immediate corollary of Theorem \ref{mainthm}.

\begin{corollary}\label{noncptcor}
	Let $\Gamma$ be a non-cocompact arithmetic lattice in $\SO(n,1)$ then there are infinitely many non-cocompact lattices in $\SL(n+1,\R)$ that contain a thin subgroup isomorphic to a finite index subgroup of $\Gamma$. \end{corollary}

In the cocompact setting, there is another construction of arithmetic lattices in $\SO(n,1)$ using quaternion algebras. However, this construction only works when $n$ is odd (again, see \cite{LiMillsonArithmetic} and \S 6.4 of \cite{WitteMorris}), which implies:

\begin{corollary}\label{cptcor}
	Let $n\geq 3$ be even and let $\Gamma$ be a cocompact arithmetic lattice in $\SO(n,1)$ then there are infinitely many cocompact lattices in $\SL(n+1,\R)$ that contain a thin subgroup isomorphic to a finite index subgroup of $\Gamma$
\end{corollary}

Our main result generalizes several previous results regarding the existence of thin groups isomorphic to hyperbolic manifolds in low dimensions. For example, there are examples of thin surface groups in both cocompact and non-cocompact lattices in $\SL(3,\R)$ \cite{LongReidThinI,LongReidThinII}. There are further examples of thin subgroups in $\SL(4,\R)$ isomorphic to the fundamental groups of closed hyperbolic 3-manifolds \cite{LongReidThinIII} and others isomorphic to the fundamental groups of finite volume hyperbolic 3-manifolds \cite{BallasLongThin}.
\subsection*{Organization of the paper}
Section \ref{convexprojgeom} provides the necessary background in convex projective geometry. Section \ref{arithmeticlattices} describes the relevant arithmetic lattices in both $\SO(n,1)$ and $\SL(n+1,\R)$. Section \ref{construction} contains the construction of the thin groups in Theorem \ref{mainthm}. Finally, Section \ref{thinness} contains the proof that the examples constructed in Section \ref{construction} are thin.

\subsection*{Acknowledgments} S.B.\  was partially supported by NSF grant DMS 1709097 and D.L.\ was partially supported by NSF grant DMS 20150301. The authors would also like to thank Alan Reid for pointing out that all non-cocompact arithmetic lattices in $\SO(n,1)$ are of orthogonal type, allowing us to weaken the hypothesis in Corollary \ref{noncptcor}. We would also like to thank the anonymous referee for several helpful suggestions that improved the paper.

\section{Convex projective geometry}\label{convexprojgeom}

Let $V=\R^{n+1}$. There is an equivalence relation on the non-zero vectors in $V$ given by $x\sim y$ if there is $\lambda>0$ such that $\lambda x=y$.  The set $S(V)$ of equivalence classes of $\sim$ is called the \emph{projective $n$-sphere}. Alternatively, $S(V)$ can be regarded as the set of rays through the origin in $V$. Sending each equivalence class to the unique representative of length 1 gives an embedding of $S(V)$ into $V$ as the unit $n$-sphere. 

The group $\GL(V)$ acts on $S(V)$, however this action is not faithful. The kernel of this action consists of positive scalar multiples of the identity, $\R^+I$. Furthermore, if $A\in \GL(V)$ then $\abs{\det(A)}^{\frac{-1}{(n+1)}}A$ has determinant $\pm 1$ and as a result we see that there is a faithful action of 
$$\SL^\pm(V)=\{A\in \GL(V)\mid \det(A)=\pm 1\}$$
 on $S(V)$. 

The projective sphere is a 2-fold cover of the more familiar \emph{projective space} $P(V)$ consisting of lines through the origin in $V$. The covering map is given by mapping a ray through the origin to the line through the origin that contains it. There is also a 2-fold covering of Lie groups from $\SL^\pm(V)$ to $\PGL(V)$ that maps an element of $\SL^\pm(V)$ to its scalar class. \red{Note that here the cover $\SL^\pm(V)$ is not connected}
\begin{figure}
\begin{center}
		\includegraphics[width=.6\linewidth]{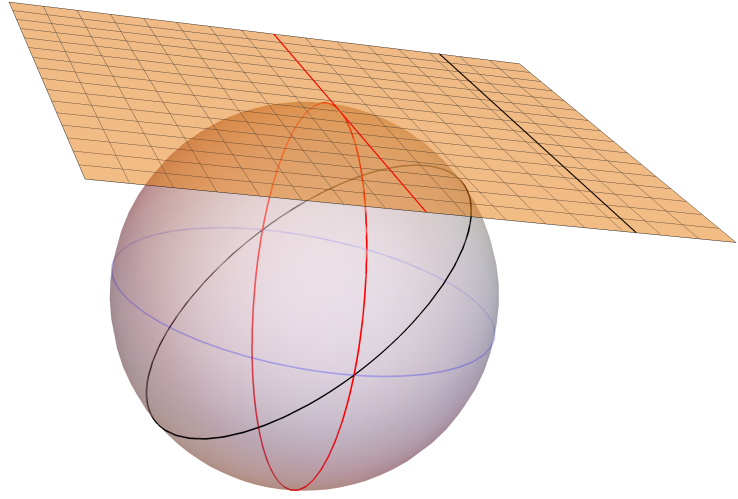}
		\caption{\label{affinepatch}The projection to an affine patch}
\end{center}
\end{figure}

Each (open) hemisphere in $S(V)$ can be identified with $\R^n$ via projection, in such a way that great circles on $S(V)$ are mapped to straight lines in $\R^n$ (see Figure \ref{affinepatch}). For this reason we refer to (open) hemispheres as \emph{affine patches} of $S(V)$ and refer to great circles as \emph{projective lines}. This identification allows us to define a notion of convexity for subsets of an affine patch. A set $\Omega\subset S(V)$ with non-empty interior is called \emph{properly convex} if its closure is a convex subset of some affine patch. \red{If in addition, $\partial \Omega$ contains no non-trivial line segments then $\Omega$ is called \emph{strictly convex}. Since $\Omega$ is convex, each point  $p\in\partial \Omega$ is contained in a hyperplane disjoint from the interior of $\Omega$. If this hyperplane is unique then $p$ is called a \emph{$C^1$ point of $\partial \Omega$.}}

 Each properly convex set $\Omega$ comes equipped with a group 
$$\red{\SL(\Omega)=\{A\in \SL^\pm(V)\mid A(\Omega)=\Omega\}.}$$ 
In other words, $\SL(\Omega)$ consists of elements of $\SL^\pm(V)$ that preserve $\Omega$. There is a similar definition for properly convex subsets of $\RP^n$ and we will allow ourselves to discuss properly convex geometry in whichever setting is more convenient. 

\red{Properly convex sets also come equipped with an $\SL(\Omega)$-invariant metric called the \emph{Hilbert metric}. If $x,y\in \Omega$ then the projective line between $x$ and $y$ intersects $\partial \Omega$ in two points $a$ and $b$ (where $a$ is chosen to be the one closer to $x$). In this context we define the Hilbert distance between $x$ and $y$ to be

$$d_\Omega(x,y)=\frac{1}{2}\log([a:x:y:b]),$$
where $[a:x:y:b]=\frac{\abs{b-x}\abs{y-a}}{\abs{x-a}\abs{b-y}}$ is the cross ratio corresponding to the projective coordinate of $y$ in the coordinate system that takes $a,x$, and $b$ to $0,1$, and $\infty$, respectively. Since projective transformations preserve cross ratios it follows that elements of $\SL(\Omega)$  are $d_\Omega$-isometries. The presence this metric ensures that discrete subgroups of $\SL(\Omega)$ act properly discontinuously on $\Omega$. 
}

To each properly convex $\Omega\subset S(V)$ it is possible to construct a \emph{dual convex set} $\Omega^\ast\subset S(V^\ast)$ defined by 
$$\Omega^\ast=\{[\phi]\in S(V^\ast)\mid \phi(v)>0\ \forall [v]\in \overline{\Omega}\}$$
It is a standard fact that $\Omega^\ast$ is a properly convex subset of $S(V)$. For each $\gamma\in \SL(\Omega)$ there is a corresponding $\gamma^\ast\in \SL(\Omega^\ast)$ given by $\gamma^\ast([\phi])=[\phi\circ\gamma^{-1}]$.  This map induces an isomorphism between $\SL(\Omega)$ and $\SL(\Omega^\ast)$. By choosing a basis for $V$ and the corresponding dual basis for $V^\ast$, it is possible to identify $\SL(V^\ast)$ and $\SL(V)$ and in these coordinates the isomorphism between $\SL(\Omega)$ and $\SL(\Omega^\ast)$ is given by $\gamma\mapsto (\gamma^{-1})^t$.

If $\Omega$ is properly convex and $\Gamma\subset \SL(\Omega)$ is discrete then $\Omega/\Gamma$ is a \emph{properly convex orbifold}. If $\Gamma$ is torsion-free then this orbifold is a manifold. By Selberg's lemma, every properly convex orbifold is finitely covered by a properly convex manifold, and for the remainder of the paper we will almost exclusively be dealing with manifolds. Furthermore, if $\Omega/\Gamma$ is a properly convex manifold then there is a corresponding \emph{dual group} $\Gamma^\ast\subset \SL(\Omega^\ast)$ and a corresponding \emph{dual properly convex manifold} $\Omega^\ast/\Gamma^\ast$. The manifolds $\Omega/\Gamma$ are diffeomorphic, but are in general not projectively equivalent.

An important example of a properly convex set is \emph{hyperbolic $n$-space}, which can be constructed as follows. Let $q$ be the quadratic form on $V$ given by the matrix
\begin{equation}\label{qform}
J_{n}=\begin{pmatrix}
	 I_n & 0 \\
	0 & -1
\end{pmatrix}.
\end{equation}
This form has signature $(n,1)$, and let $\mathcal{C}_q$ be a component of  the cone $\{v\in V\mid q(v)<0\}$. The image of $\mathcal{C}_q$ in $S(V)$ gives a model of hyperbolic space called the \emph{Klein model} of hyperbolic space which we denote $\HH^n$. In this setting, \red{$d_\Omega$ is the standard hyperbolic metric and} $\SL(\HH^n)$ is equal to the group $O(J_n)^+$ of elements of $\SL^\pm(V)$ that preserve both $J_n$ and $\mathcal{C}_q$.  When $\Omega=\HH^n$ and $\Gamma\subset \SL(\HH^n)$ is a discrete, torsion-free group then $\Omega/\Gamma$ is a \emph{complete hyperbolic manifold}. \red{It is a standard fact that if $\Omega/\Gamma$ is a complete hyperbolic manifold then the dual properly convex manifold, $\Omega^\ast/\Gamma^\ast$ is projectively equivalent to $\Omega/\Gamma$, with the projective equivalence being induced by the map from $V$ to $V^\ast$ induced by $q$.} 

If $N$ is an orientable manifold then a \emph{properly convex structure} on $N$ is a pair $(\Omega/\Gamma,f)$ where $\Omega/\Gamma$ is a properly convex manifold and $f:N\to \Omega/\Gamma$ is a diffeomorphism. The map $f$ induces an isomorphism $f_\ast:\pi_1N\to \Gamma$. Since $\Gamma\subset \SL^\pm(V)$ we can regard $f_\ast$ as a representation from $\pi_1N$ into the Lie group $\SL^\pm(V)$ which we call the \emph{holonomy} of the structure $(\Omega/\Gamma,f)$. Since $N$ is orientable it is easy to show that the holonomy always has image in $\SL(V)$. Observe that by definition, the holonomy is an isomorphism between $\pi_1N$ and $\Gamma$, and it follows immediately that the holonomy representation is injective.

Given a properly convex structure $(\Omega/\Gamma,f)$ on $N$ and an element $g\in \SL^\pm(V)$ it is easy to check that $g:\Omega\to g(\Omega)$ induces a diffeomorphism $\overline{g}:\Omega/\Gamma\to g(\Omega)/g\Gamma g^{-1}$ and that $(g(\Omega)/g\Gamma g^{-1},\overline{g}\circ f)$ is also a properly convex structure on $N$. Furthermore, the holonomy of this new structure is obtained by post-composing $f_\ast$ by conjugation  in $\SL^\pm(V)$ by $g$. Two properly convex structures $(\Omega/\Gamma,f)$ and $(\Omega'/\Gamma',f')$ on $N$ are \emph{equivalent} if there is $g\in \SL^\pm(V)$ such that $\Omega'/\Gamma'=g(\Omega)/g\Gamma g^{-1}$, and $f'$ is isotopic to $\overline{g}\circ f$. 

\subsection{Generalized cusps}\label{gencusps}
A generalized cusp is a certain type of properly convex manifold that generalizes a cusp in a finite volume hyperbolic manifold. Specifically, a properly convex \red{$n$}-manifold $C\cong\Omega/\Gamma$ is a \emph{generalized cusp} if $\Gamma$ is a virtually abelian and $C\cong \partial C\times (0,\infty)$ with $\partial C$ a compact strictly convex submanifold of $C$. \red{In this context, $\partial C$ being strictly convex means that for each $p\in \partial C$ there is a projective hyperplane $H_p$ and a neighborhood $U_p$ of $p$ so that $H_p\cap U_p=\{p\}$.}  Such manifolds were recently classified by the first author, D.\ Cooper, and A. Leitner \cite{BCL}. One consequence of this classification is that for \red{$n$-dimensional projective manifolds} there are $n+1$ different \emph{types} of generalized cusps. For the purposes of this work only two of these types (type 0 and type 1) will arise. We will also restrict to cusps with the property that $\partial C$ is diffeomorphic to an $(n-1)$-torus. \red{Such cusps will be called \emph{torus cusps}} and we now briefly describe these types of cusps. 

Let 
$$\Omega_0=\left\{[x_1\dots:x_{n+1}]\in P(V)\mid x_1x_{n+1}>\frac{1}{2}\left(x_2^2+\ldots+x_n^2\right)\right\}.$$ It is not difficult to see that $\Omega_0$ is projectively equivalent to the Klein model for hyperbolic space. Let $P_0$ be the collection (of equivalence classes) of matrices with block form
\begin{equation}\label{type0}
	\begin{pmatrix}
	1 & v & \frac{1}{2}\abs{v}^2\\
	0 & I_{n-1} & v^t\\
	0 & 0 & 1
\end{pmatrix},
\end{equation}

where $v$ is a (row) vector in $\R^{n-1}$, $I_{n-1}$ is the identity matrix and the zeros are blocks of the appropriate size to make \eqref{type0} a $(n+1)\times (n+1)$ matrix. A simple computation shows that the elements of $P_0$ preserve $\Omega_0$ (they are just the parabolic isometries of $\HH^n$ that fix $\infty=[1:0\ldots:0]$). There is a foliation of $\Omega_0$ by strictly convex hypersurfaces of the form 
$$\mathcal{H}_c=\left\{[x_1:\ldots:x_{n}:1]\mid x_1-\frac{1}{2}(x^2_2+\ldots+x_{n}^2)=c\right\},$$
\begin{figure}
\begin{center}
	\includegraphics[scale=.4]{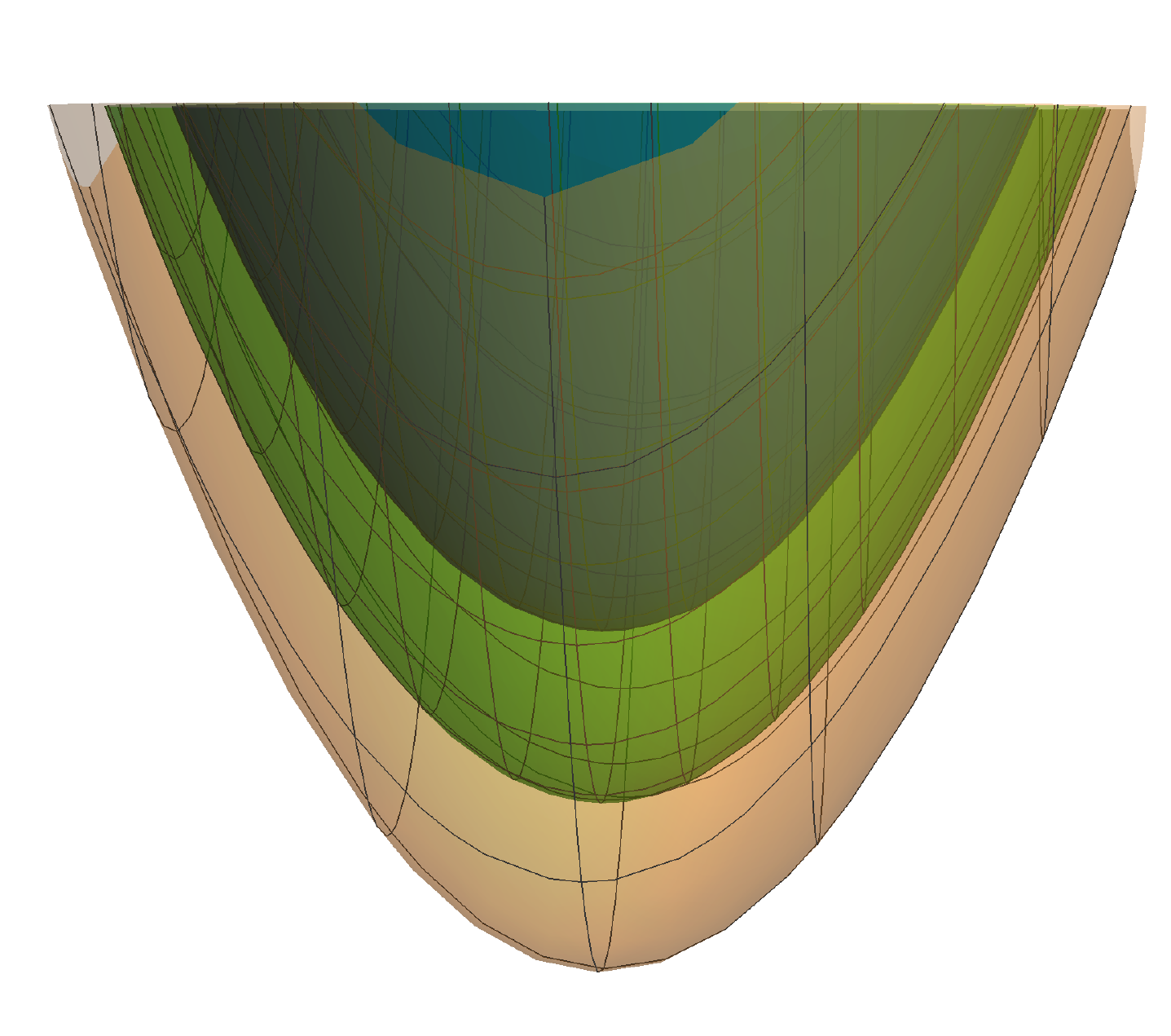}
		\caption{\label{type0cusp} The domain $\Omega_0$ and its foliation by horospheres}
\end{center}
\end{figure}
 for $c>0$ whose leaves are preserved setwise by $P_0$. In terms of hyperbolic geometry the $\mathcal{H}_c$ are  \emph{horospheres} centered at $\infty$ and the convex hull of a leaf is a \emph{horoball} centered at $\infty$. The group $P_0$ is isomorphic to $\R^{n-1}$ and so if $\Gamma\subset P_0$ is a lattice then $\Gamma$ is isomorphic to $\Z^{n-1}$ and the quotient $\Omega/\Gamma$ is a \emph{generalized (torus) cusp of type 0}.

Next, let 
$$\Omega_1=\left\{[x_1:\ldots:x_{n+1}]\mid x_1x_{n+1}>-\log \abs{x_2}+\frac{1}{2}(x_3^2+\ldots+x^2_{n}),\  x_2x_{n+1}>0\right\}$$
 and let $P_1$ be the collection (of equivalence classes) of matrices of block form 
\begin{equation}\label{type1}
\begin{pmatrix}
	1 & 0 & v & -u+\frac{1}{2}\abs{v}^2\\
	0 & e^u & 0 & 0\\
	0 & 0 & I_{n-2} & v^t\\
	0 & 0 & 0 & 1
	\end{pmatrix},
\end{equation}
where $u\in \R$, $v\in \R^{n-2}$, $I_{n-2}$ is the identity matrix and the zeros are the appropriate size to make \eqref{type1} an $(n+1)\times (n+1)$ matrix. Again, it is easy to check that $P_1$ preserves $\Omega_1$. \red{Elements of $P_1$ for which $u=0$ are called \emph{parabolic} and every parabolic element preserves each copy of $\HH^{n-1}$ obtained by intersecting $\Omega_1$ and the plane $x_2=d$ with $d>0$. The domain $\Omega_1$ contains a unique line segment $\ell_\infty$ with endpoints $q_+$ and $q_-$ in its boundary. In the coordinates we have chosen $q_+=[e_1]$ and $q_-=[e_2]$. These points can distinguished by the fact that $q_-$ is a $C^1$ point and $q_+$ is not. The group $P_1$ preserves $\ell_\infty$ and the parabolic elements fix $\ell_\infty$ pointwise}

 Again, there is a foliation of $\Omega_1$ by strictly convex hypersurfaces of the form 
$$\mathcal{H}_c=\left\{[x_1:\ldots:x_n:1]\mid x_1+\log x_2-\frac{1}{2}(x_3^2+\ldots+x_n^2)=c,\ x_2>0\right\}$$
 for $c>0$ that is preserved by $P_1$. Again, each leaf is a $P_1$ orbit, we call the leaves of this foliation \emph{horospheres} and call the convex hulls of a leaves \emph{horoballs}. Again $P_1\cong \R^{n-1}$ and if $\Gamma\subset P_1$ is a lattice then $\Gamma\cong \Z^{n-1}$ and $\Omega_1/\Gamma$ is a \emph{generalized (torus) cusp of type 1}. For the remainder of this paper when we say generalized cusp that will mean a generalized torus cusp of type 0 or type 1.

\begin{figure}
\begin{center}
	\includegraphics[scale=.4]{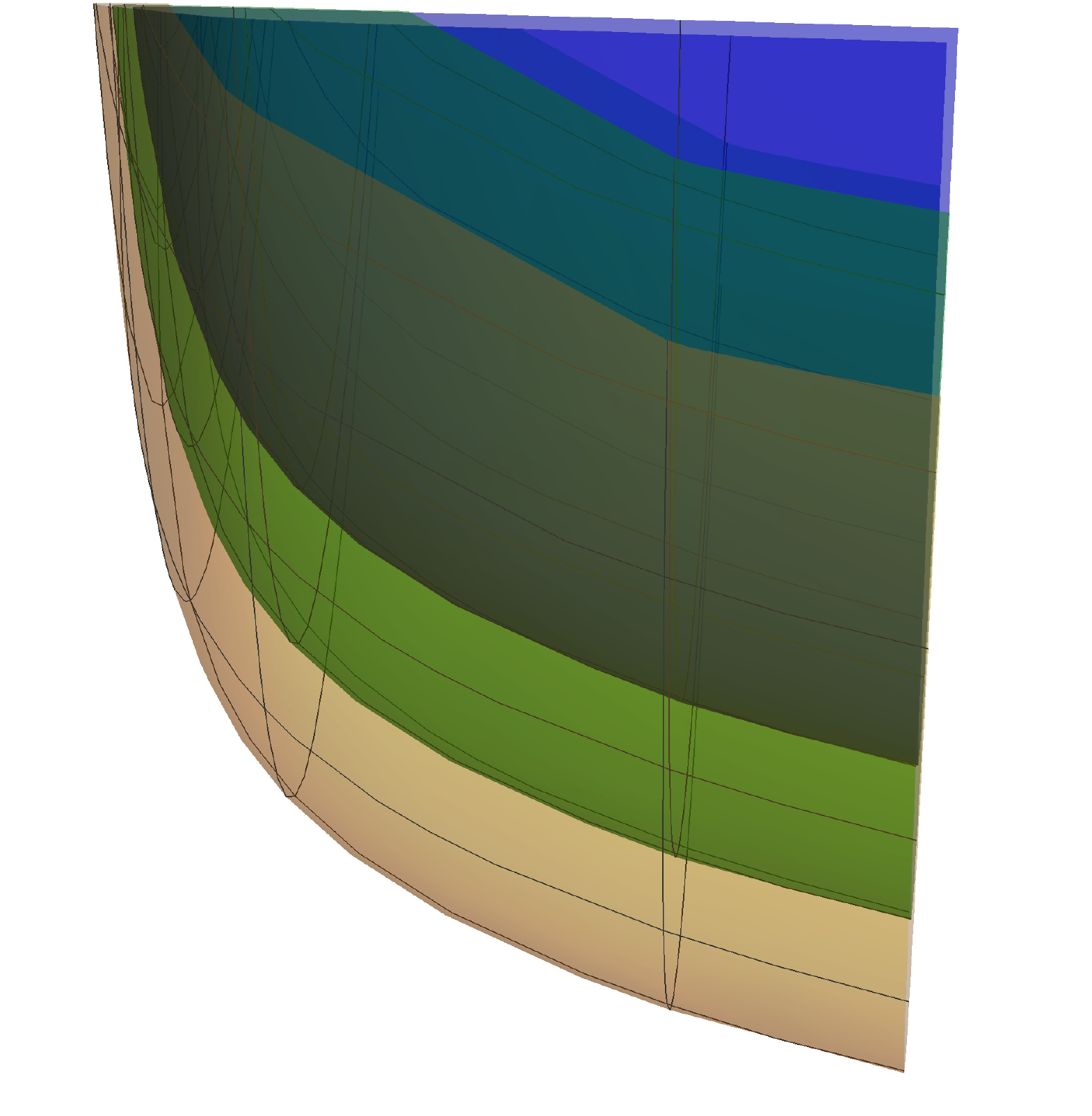}
		\caption{\label{type1cusp} The domain $\Omega_1$ and its foliation by horospheres}
\end{center}
\end{figure}

Generalized cusps of a fixed type are closed under two important operations: taking finite sheeted covers and duality. If $\Omega/\Gamma$ is a generalized cusp then taking a finite sheeted cover corresponds to choosing a finite index subgroup $\Gamma'\subset \Gamma$. The group $\Gamma'$ is also a lattice in $P_0$ or $P_1$ and hence $\Omega/\Gamma'$ is a generalized cusp. The fact that generalized cusps are closed under duality follows immediately from the observation that the group $P_0^t$ (resp.\ $P_1^t$) obtained by taking the transpose of the elements of $P_0$ (resp.\ $P_1$) is conjugate to $P_0$ (resp.\ $P_1$).  

One distinction between these two types of cusps that will be important for our purposes in Section \ref{thinness} is that the group $P_0$ is Zariski closed, but the group $P_1$ is not. The Zariski closure, $\overline{P_1}$ of $P_1$ is $n$-dimensional and consists of matrices of the form
\begin{equation}\label{p1zclosure}
	\begin{pmatrix}
	1 & 0 & v & w\\
	0 & u & 0 & 0\\
	0 & 0 & I_{n-2} & v^t\\
	0 & 0 & 0 & 1
\end{pmatrix},
\end{equation}
where $u\neq 0$, $w\in \R$, and $v\in \R^{n-2}$. Furthermore, we have the following lemma describing the generic orbits of $\overline{P_1}$ whose proof is a straightforward computation.

\begin{lemma}\label{p1barorbits}
	If $x\notin \ker(e_2^\ast)\cup\ker(e_{n+1}^\ast)$ then $\overline{P_1}\cdot x$ is open in $\RP^n$.
\end{lemma}


\subsection{Bending}\label{bending}

We now describe a construction that allows one to start with a (special) hyperbolic manifold and produce a family of inequivalent convex projective structures.

Suppose that $M=\HH^n/\Gamma$ is a complete, finite-volume hyperbolic manifold, and suppose that $M$ contains an embedded totally geodesic hypersurface, $\Sigma$. There is an embedding of $\SO(J_{n-1})$ into $\SO(J_n)$ via the embedding

$$\SO(J_{n-1})\mapsto \begin{pmatrix}
 1 & 0\\
 0 & \SO(J_{n-1})	
 \end{pmatrix}.
$$ 
Under this embedding, the image of $\SO(J_{n-1})$ stabilizes a copy of $\HH^{n-1}$ in $\HH^n$ and $\Sigma\cong\HH^{n-1}/\Lambda$, where $\Lambda$ is a subgroup of $\SO(J_{n-1})\cap \Gamma$. For each $t\in \R$, the element

$$B_t=\begin{pmatrix}
	e^{-nt} & \\
	 & e^{t}I_{n}
\end{pmatrix}$$
centralizes $\SO(J_{n-1})$ and hence centralizes $\Lambda$. 

Let $N=M$ and let $id:N\to M$ be the identity, then $(M,id)$ is a convex projective structure on $N$. Let $\rho:\pi_1N\to \SL(V)$ be the holonomy of this structure. Concretely, $\rho$ is just the inclusion of $\pi_1N\cong \Gamma$ into $\SL(V)$. We now define a family, $\rho_t:\pi_1N\to \SL(V)$, of representations such that $\rho_0=\rho$. The construction depends on whether or not $\Sigma$ is separating. 

If $\Sigma$ is separating then $\Gamma$ splits as an amalgamated product $\Gamma_1\ast_{\Lambda}\Gamma_2$, where the $\Gamma_i$ are the fundamental groups of the components of $M\bs \Sigma$. Then $\rho_t$ is defined by the property that $\rho_t(\gamma)=\rho(\gamma)$ if $\gamma\in \Gamma_1$ and $\rho_t(\gamma)=B_t\rho_0(\gamma)B_t^{-1}$ if $\gamma\in \Gamma_2$. Since $B_t$ centralizes $\Lambda$ this gives a well defined representation $\rho_t:\pi_1N\to \SL(V)$. 
	
In the separating case, $\Gamma=\Gamma'\ast_{s}$ is an HNN extension where $\Gamma'$ is the fundamental group of $M\bs \Sigma$. In this case $\rho_t$ is defined by the property that $\rho_t(\gamma)=\rho(\gamma)$ if $\gamma\in \Gamma'$ and $\rho_t(s)=B_t\rho(s)$. Again it is easy to see that since $B_t$ centralizes $\Lambda$ that this gives a well defined representation $\rho_t:\pi_1N\to \SL(V)$.  

In either case we say that the family of $\rho_t$ is \emph{obtained by bending $M$ along $\Sigma$}. From the construction, it is not obvious that the representations $\rho_t$ are the holonomy of a convex projective structure. However, the following theorem  guarantees that this is the case

\begin{theorem}[See \cite{Ko,MarquisBending}]\label{Marquisbending}
	For each $t\in \R$ the representation $\rho_t$ obtained by bending $M$ along $\Sigma$ is the holonomy of a properly convex projective structure on $N$. 
\end{theorem}

\red{\begin{remark}\label{dual_bending_is_bending}
The property of being obtained from bending is closed under two important operations: taking finite sheeted covers and duality. First, if $M=\Omega/\Gamma$ and $M'$ is a finite sheeted cover of $M$ then $M'$ is of the form $\Omega/\Gamma'$ where $\Gamma'$ is a finite index subgroup of $\Gamma$. If $M$ is obtained by bending a finite volume hyperbolic manifold $N$ along an embedded totally geodesic hypersurface $\Sigma$ then $M'$ is obtained by similtaneously bending the cover $N'$ of $N$ corresponding to $M'$ along the (possibly disjoint) totally geodesic embedded hypersurface $\Sigma'$ obtained by taking the complete preimage of $\Sigma$ in $N'$.

 If $N=\HH^n/\Gamma_0$ is a complete hyperbolic manifold containing a totally geodesic hypersurface $\Sigma$, then its dual projective manifold, $N^\ast$, is projectively equivalent to $N$, and hence also contains a totally geodesic hypersurface, $\Sigma^\ast$. If $M=\Omega/\Gamma$ is obtained from bending $N$ along $\Sigma$, then the dual projective manifold, $M^\ast$ is obtained by bending $N^\ast$ along $\Sigma^\ast$. 
\end{remark}
}

The following theorem from \cite{BalMar} addresses which types of cusps arise when one bends a hyperbolic manifold along a totally geodesic hypersurface.

\begin{theorem}[Cor.\ 5.10 of \cite{BalMar}]\label{balmarbending}
	Let $M$ be a finite volume hyperbolic manifold and let $\Sigma$ be an embedded totally geodesic hypersurface. If $M'$ is the properly convex manifold obtained by bending $M$ along $\Sigma$ then each end of $M$ is a generalized cusp of type 0 or type 1. 
\end{theorem}

\subsection{Properties of the holonomy}\label{holonomy}

In this section we discuss some important properties of the holonomy representation  of convex projective structures \red{that arise from bending.} A representation $\rho:\Gamma\to \GL(V)$ is called \emph{strongly irreducible} if its restriction to any finite index subgroup is irreducible. The main result of this section is the following:

\begin{theorem}\label{gencupsstrirred}
	Let $(\Omega/\Gamma,f)$ be a convex projective structure on $M$ and let $\rho$ be its holonomy. If $\Omega/\Gamma$ \red{is obtained by bending a finite volume hyperbolic manifold along an embedded totally geodesic hypersurface}
	then $\rho$ is strongly irreducible. 
\end{theorem}

Before proceeding with the proof of Theorem \ref{gencupsstrirred}, we need a few lemmas. If $P$ is a subset contained in some affine patch in $S(V)$ then let $\hull(P)$ denote the convex hull of $P$ (note that since $P$ is contained in an affine patch that this is well defined). 

\begin{lemma}\label{gencusphull}
		Suppose that $M=\Omega/\Gamma$ is a properly convex manifold \red{obtained from bending a finite volume hyperbolic manifold along an embedded totally geodesic hypersurface}
		then for any $p\in \overline{\Omega}$, $\hull(\Gamma\cdot p)$ has non-empty interior.
\end{lemma}

\begin{proof}
		\red{If $M$ is closed then the result follows from \cite[Prop.\ 3]{Vey}, and so we assume that $M$ has at least 1 cusp, which by Theorem \ref{balmarbending} is a generalized cusp of type 0 or type 1. Let $\Delta$ be the fundamental group of one of the generalized cusps. By \cite[Lem.\ 5.7]{BalMar} we can find horoballs $\mathcal{H}$ and $\mathcal{H}'$ so that (after conjugating in $\SL(V)$) $\mathcal{H}\subset \Omega\subset \mathcal{H}'$. It follows that there is a unique projective hyperplane $L$ with the property that if $p\in \overline{\Omega}\backslash L$ then $\hull(\Lambda\cdot p)$ contains a horoball. In particular, for such $p$, $\hull(\Gamma\cdot p)$ has non-empty interior. In the coordinates of the previous section $L$ is the projective hyperplane coming from $\ker(e_{n+1}^\ast)$. Furthermore, $\overline{\Omega}\cap L$ is either the point $\infty$ if the cusp is type 0 or the line segment $\ell_\infty$ from the previous section if the cusp is type 1.

		In light of this, the proof will be complete if we can show that for each $p\in \overline{\Omega}$ the orbit $\Gamma\cdot p$ contains a point in $\overline{\Omega}\backslash L$.  Suppose that $p\in \overline{\Omega}\cap L$. Since $M$ is obtained by bending, it contains a subgroup $\Lambda$ corresponding to the fundamental group of the totally geodesic hypersurface. This subgroup preserves a copy of $(n-1)$-dimensional hyperbolic space, $H_\Lambda\subset \Omega$ and fixes a unique point $p_\infty\in P(V)$ dual to $H_\Lambda$. It follows that if $p\neq p_\infty$ then the $\Lambda$ orbit of $p$ accumulates to any point in $\partial H_\Lambda$. Since the plane $L$ is a supporting plane for $\Omega$ and the plane containing $H_\Lambda$ meets the interior of $\Omega$ it follows that there is a point of $\partial \Omega$ that is not contained in $L$ and hence a $g\in \Lambda$ so that $g\cdot p\notin L$. 
		
		This leave only the case where $p=p_\infty$. In this case let $g\in \Lambda$ be a hyperbolic isometry, let $p_0\in \partial H_\Lambda$ be its repelling fixed point, and let $h\in \Delta$ be parabolic (since $\Omega/\Gamma$ came from bending such an element is guaranteed to exist). Let $\ell$ be the projective line connecting $p_\infty$ and $p_0$. Since $p_\infty\in \ell_{\infty}$ and $h$ is parabolic it follows that $h\cdot p_\infty=p_\infty$ and so $\ell$ and $h\cdot \ell$ are contained in a projective 2-plane, $L'$. Let $\Omega'=\Omega\cap L'$, then $p_\infty$ is a $C^1$ point of $\partial \Omega'$. To see this observe that unless the cusp is type 1 and $p_\infty=q_+$, $p_\infty$ is already a $C^1$ point of $\partial \Omega$, and thus a $C^1$ point of $\partial \Omega'$. On the other hand, if $p_\infty=q_+$ then $p_\infty$ is not a $C^1$ point of $\partial \Omega$, however, $\mathcal{H}\cap L'\subset \Omega'\subset \mathcal{H}'\cap L'$. Both $\mathcal{H}\cap L'$ and $\mathcal{H}'\cap L'$ are projectively equivalent to copies of $\HH^2$ and meet at $p_\infty$, and so $p_\infty$ is a $C^1$ point of $\partial \Omega'$.  
		
		Since $p_\infty\in \partial \mathcal{H}$ and $\mathcal{H}\subset \Omega\subset \mathcal{H}'$, it follows that the line $\ell$ intersects the interior of $\Omega$. Let $x\in \ell\cap \text{int}(\Omega)$ and observe that for each natural number $n$, $g^n\cdot x\in \ell$, $hg^n\cdot x\in h\cdot \ell$, and both sequences limit to $p_\infty$ as $n\to \infty$ since $p_0$ is the repelling fixed point of $g$.  Since $p_\infty$ is a $C^1$ point of $\partial \Omega'$, it follows from \cite[Prop.\ 3.4 (H7)]{CoLoTi} that $d_n:=d_{\Omega'}(g^n\cdot x,hg^n\cdot x)\to 0$ as $n\to \infty$. However, $\Omega'$ is a totally geodesic subspace of $\Omega$ (with respect to $d_\Omega$), and so this implies that $d_\Omega(g^n\cdot x,hg^n\cdot x)\to 0$ as $n\to\infty$. It follows that $d_\Omega(x,g^{-n}hg^n\cdot x)\to 0$ as $n\to\infty$ but this is a contradiction since the group $\Gamma$ acts properly discontinuously on $\Omega$.

		  }

		\end{proof}

The following lemma is the basis for the proof of Theorem \ref{gencupsstrirred}. The lemma and its proof are inspired by a similar result of J.\ Vey \cite[Prop.\ 4]{Vey}.

\begin{lemma}\label{irreducible}
		Suppose that $\Omega\subset P(V)$ is properly convex and that $\Gamma\subset \SL(\Omega)$ is a group with the property that for every $p\in \overline{\Omega}$, $\hull(\Gamma\cdot p)$ has non-empty interior. If $L$ is a $\Gamma$-invariant subspace of $V$ and $P(L)\cap \overline{\Omega}\neq \emptyset$ then $L=V$.   
	\end{lemma}
	\begin{proof}
	Let $L\subset V$ be a $\Gamma$-invariant subspace such that $P(L)\cap \overline{\Omega}\neq \emptyset$, and let $p$ be a point in the intersection. Since $p\in \overline{\Omega}$ it follows that $\hull(\Gamma\cdot p)$ has non empty interior. Furthermore, since $p\in L$ and $L$ is both $\Gamma$-invariant and convex it follows that $\hull(\Gamma\cdot p)\subset P(L)$. Since $\hull(\Gamma\cdot p)$ has non-empty interior so does $P(L)$. It follows that $L=V$.  	
	\end{proof}
	
	We can now prove Theorem \ref{gencupsstrirred}
	
	\begin{proof}[Proof of Theorem \ref{gencupsstrirred}]
		Suppose that $L\subset V$ is a $\Gamma$-invariant subspace. First assume that $P(L)\cap \overline{\Omega}\neq \emptyset$. Combining Lemmas  \ref{gencusphull} and \ref{irreducible} it follows that $L=V$. On the other hand, suppose that $L\cap \overline{\Omega}=\emptyset$ then $L$ corresponds to a non-trivial subspace $L^\ast\subset V^\ast$ such that $P(L^\ast)\cap \overline{\Omega^\ast}\neq \emptyset$. \red{Since $\Omega/\Gamma$ is obtained from bending it follows that $\Omega^\ast/\Gamma^\ast$ is also obtained from bending a finite volume manifold along an embedded totally geodesic hypersurface (see Remark \ref{dual_bending_is_bending})} and so we can apply the same argument as before to show that $L^\ast=V^\ast$. It follows that $L=0$, and so there are no proper non-trivial $\Gamma$-invariant subspace. Hence $\Gamma$ acts irreducibly on $V$. 
		
		Finally, if $\Gamma'$ is a finite index subgroup of $\Gamma$ then  \red{$\Omega/\Gamma'$ is a properly convex manifold that also arises from bending a finite volume hyperbolic manifold along an embedded totally geodesic hypersurface (again, see Remark \ref{dual_bending_is_bending})}, and so by the argument above $\Gamma'$ also acts irreducibly on $V$. 

	\end{proof}

\subsection{Zariski closures and limit sets}\label{zarclosure}
We close this section by describing some properties of the Zariski closure of the groups obtained by bending. Before proceeding we introduce some terminology and notation. Let $g\in \SL(V)$ then $g$ is \emph{proximal} if $g$ has a unique (counted with multiplicity) eigenvalue of maximum modulus. It follows that this eigenvalue must be real and that $g$ is proximal if and only if $g$ has a unique attracting fixed point for its action on $P(V)$. If $G$ is a subgroup of $\SL(V)$ then $G$ is \emph{proximal} if it contains a proximal element. 

If $G\subset \SL(V)$ is a group then we define the \emph{limit set} of $G$, denoted $\Lambda_G$ as 
$$\Lambda_G=\overline{\{x\in P(v)\mid \textup{$x$ a fixed point of some proximal $g\in G$}\}}$$

By construction, this $\Lambda_G$ is closed and if $G$ is proximal then $\Lambda_G$ is non-empty. In this generality the limit set was introduced by Goldscheid--Guivarc'h \cite{GoldGuiv} and this construction reduces to the more familiar notion of limit set when $G$ is a Kleinian group. The limit set has the following important properties. 

\begin{theorem}[Thm.\ 2.3 of \cite{GoldGuiv}]\label{minimal}
	If $G$ is proximal and acts irreducibly on $V$ then $\Lambda_G$ is the unique minimal non-empty closed $G$-invariant subset of $P(V)$.
\end{theorem}

Next, let $M=\HH^n/\Gamma$ be a finite volume (non-compact) hyperbolic manifold containing an embedded totally geodesic hypersurface $\Sigma$, let $\Gamma_t=\rho_t(\Gamma)$ be the group obtained by bending $M$ along $\Sigma$, and let $G_t$ be the Zariski closure of $\Gamma_t$. The following lemma summarizes some properties of $G_t$ and its relation to $\Lambda_G$.

\begin{lemma}\label{limitset}
	Let $\rho_t$ be obtained by bending $M$ along $\Sigma$, let $\Gamma_t=\rho_t(\Gamma)$ and let $G_t$ be the Zariski closure of $\Gamma_t$ then 
	\begin{itemize}
		\item The identity component, $G^0_t$, of  $G_t$ is semisimple, proximal, and acts irreducibly on $V$
		\item $\Lambda_{G^0_t}=G^0_t\cdot x$ for any $x\in \Lambda_{G^0_t}$. 
	\end{itemize}
\end{lemma}

\begin{proof}
	 The group $G^0_t$ is a finite index subgroup of $G_t$ and contains the group $G^0_t\cap \Gamma_t$ which has finite index in $\Gamma_t$. By Theorem \ref{gencupsstrirred} it follows that $G^0_t\cap \Gamma_t$ and hence $G^0_t$ acts irreducibly on $V$, \red{and so $V$ becomes a simple $\R[G^0_t]$ module. Let $R_t$ be the unipotent radical of $G_t^0$ and let $V_\C$ be the complexification of $V$. Since $R_t$ is unipotent and solvable the Lie-Kolchin theorem implies that  there is a non-trivial $\C[R_t]$-submodule, $E_\C$, of $V_\C$ consisting of simultaneous 1-eigenvalues of $R_t$. The submodule $E_\C$ is conjugation invariant and so there is a non-trivial $\R[R_t]$ submodule, $E_\R$ of $V$ whose complexification is $E_\C$. Furthermore, since $R_t$ is normal in $G_t^0$ it follows that $E_\R$ is also a $\R[G_t^0]$ submodule. By simplicity, it follows that $E_\R$ equals $V$, and so $R_t$ acts trivially on $V$, and is thus trivial. Hence $G^0_t$ is reductive.  }
	 
	  The group $\rho_0(\pi_1\Sigma)$ is easily seen to contain a proximal element and by construction $\rho_t(\pi_1\Sigma)=\rho_0(\pi_1\Sigma)$. It follows that $\Gamma_t$ \red{(and hence $G^0_t$) contains a proximal element $g$. Next, suppose that $h$ is an element in the center of $G^0_t$. The element $g$ has a 1-dimensional real eigenspace $V_g\subset V$. Since $h$ is central it preserves $V_g$ and thus also has a real eigenspace $V_h$ (possibly of dimension larger than 1). However, since $h$ is central, $V_h$ is also $G^0_t$-invariant, which implies that $V_h=V$ and so $h$ is a scalar matrix. Since $G^0_t\subset \SL(V)$, we must have $h=\pm I$. It follows that the center of $G^0_t$ is discrete. Since $G^0_t$ is reductive, its radical is a connected subgroup of its center, and so the radical is actually trivial. Hence $G^0_t$ is also semisimple}.
	  	
	Next, let $G^0_t=KAN$ be an Iwasawa decomposition of $G^0_t$. Since \red{$G^0_t$} is proximal \red{it follows from \cite[Thm 6.3]{AbMarSo}}, that $N$ has a unique global fixed point $x_N\in P(V)$, which is a weight vector for the highest weight of $G^0_t$ with respect to this decomposition. Since $A$ normalizes $N$ it follows that $A$ also preserves $x_N$, and so $G^0_t\cdot x_N=K\cdot x_N$ is a closed orbit, (since $K$ is compact). Furthermore, it is easy to see that $x_N\in \Lambda_{G^0_t}$ and so $G\cdot x_N$ is a closed $G^0_t$-invariant subset of $\Lambda_{G^0_t}$. Therefore, by Theorem \ref{minimal}, $G^0_t\cdot x_N=\Lambda_{G^0_t}$. Finally, an orbit is the orbit of any of its points and so it follows that if $x\in \Lambda_{G^0_t}$ then $\Lambda_{G^0_t}=G^0_t\cdot x$.
\end{proof}

\section{Arithmetic lattices}\label{arithmeticlattices}

Up until now we have been implicitly working over the real numbers. In this section we will have to work with other fields and rings and we would like this to be explicit in our notation. For this reason when we discuss groups of matrices we will need to explicitly specify where the entries lie. Henceforth, we will denote $\SO(J_n)$ as $\SO(n,1)$.  

Let $F$ be a number field and recall that $F$ is \emph{totally real} if every embedding $\sigma:F\to \C$ has the property that $\sigma(F)\subset \R\subset \C$. By choosing one of these embeddings we will regard $F$ as a subfield of $\R$. If $\alpha\neq 0$ is an element of a totally real field then define $s(\alpha)$ to be the number of non-identity embeddings $\sigma:F\to \R$ for which $\sigma(\alpha)>0$. 

\subsection{Lattices in $\SO(n,1)$}\label{lattices_in_son1}

There are multiple constructions that give rise to different classes of arithmetic lattices in $\SO(n,1)$. We now explain the simplest of these constructions and the only one that will be relevant for our purposes.  

Let $F$ be a totally real number field, let $\O_F$ be its ring of integers and suppose we have chosen $\alpha_1,\ldots,\alpha_n$ be positive elements of \red{$\mathcal{O}_F$} such that $s(\alpha_i)=0$ (i.e. the $\alpha_i$ are negative under all other embeddings of $F$). Let $\vec\alpha=(\alpha_1,\ldots,\alpha_n)$ and define $J^{\vec\alpha}=diag(\alpha_1,\ldots,\alpha_n,-1)$. Next, let $\mathbb{X}\in \{\R,F,\O_F\}$ and  define the groups $\SO(J^{\vec\alpha},\mathbb{X})=\{A\in \SL(n+1,\mathbb{X})\mid A^tJ^{\vec\alpha}A=J^{\vec\alpha}\}$. It is well known that $\SO(J^{\vec\alpha},\O_F)$ is a lattice in $\SO(J^{\vec\alpha},\R)$ \red{(see \cite[\S 6.4]{WitteMorris}, particularly Prop.\ 6.4.4 for a detailed explanation)}. Furthermore, the forms $J^{\vec\alpha}$ and $J_n$ are $\R$-equivalent \red{and so $\SO(J^{\vec{\alpha}},\R)$ and $\SO(n,1)$ are conjugate Lie groups} and so $\SO(J^{\vec\alpha},\O_F)$ \red{is conjugate to} a lattice in $\SO(n,1)$. Hence we can regard $\HH^n/\SO(J^{\vec\alpha},\O_F)$ as a hyperbolic orbifold. The lattices constructed in this fashion are cocompact if and only if $F\neq \Q$. A lattice in $\SO(n,1)$ that is commensurable with $\SO(J^{\vec\alpha},\O_F)$ for some choice of $F$ and $\vec\alpha$ is called an \emph{arithmetic lattice of orthogonal type}.

If $\tilde\Gamma=\SO(J^{\vec \alpha},\O_F)$ constructed above, then $O=\HH^n/\tilde \Gamma$ will contain several immersed totally geodesic hypersurfaces, and we now describe one of them and show how it can be promoted to an embedded totally geodesic hypersuface with nice intersection properties in a finite sheeted manifold cover of $O$. Specifically, let $\vec\alpha_1=(\alpha_2\ldots,\alpha_n)$, then $\tilde\Gamma_1=\SO(J^{\vec\alpha_1},\O_F)$ embeds reducibly in $\SO(J^{\vec\alpha},\O_F)$ via

$$\SO(J^{\red{\vec\alpha_1}},\O_F)\hookrightarrow\begin{pmatrix}
	1 & \\
	& \SO(J^{\red{\vec\alpha_1}},\O_F)
\end{pmatrix}$$

 Furthermore, $\tilde\Gamma_1$ is (commensurable with) a lattice in $\SO(n-1,1)$. The obvious embedding of $\tilde \Gamma_1$ into $\tilde \Gamma$ induces an immersion of $\HH^{n-1}/\tilde \Gamma_1$ in $\HH^n/\tilde \Gamma$. By combining results of Bergeron \cite{Bergeron}, and Selberg's Lemma we can find finite index subgroups $\Gamma$ (resp.\ $\Gamma_1$) so that $M=\HH^n/\Gamma$ (resp.\ $M_1=\HH^{n-1}/\Gamma_1$) is a manifold and $M_1$ is an embedded totally geodesic hypersurface in $M$. Furthermore, if $M$ is noncompact, then by using the argument from \cite[Thm 7.1]{BalMar} it is possible pass to a further finite cover of $M$ where all the cusps \red{are torus cusps} and the intersection of $M_1$ with one of the \red{cusps is connected}. Shortly we will bend $M$ along $M_1$ in order to produce thin subgroups in lattices in $\SL(n+1,\R)$.

\subsection{Lattices in $\SL(n+1,\R)$}

Next, we describe the lattices in $\SL(n+1,\R)$ in which we will construct thin subgroups. The construction is similar to the one in the previous section, and can be thought of as its ``unitary'' analogue. 

Again, let $F$ be a totally real number field, let $\O_F$ be its ring of integers, and suppose we have chosen $\alpha_1,\ldots,\alpha_n$ to be positive elements of $\mathcal{O}_F$ such that $s(\alpha_i)=0$. Next, let $L$ be a real quadratic extension of $F$ and let $\O_L$ be the ring of integers of this number field. $L$ is a quadratic extension of $F$ and so there is a unique non-trivial Galois automorphism of $L$ over $F$ that we denote $\tau:L\to L$.

If $M$ is a matrix with entries in $L$ then the \emph{conjugate transpose of $M$ (over $L$)}, denote $M^\ast$ is the matrix obtained by taking the \red{transpose} of $M$ and applying $\tau$ to its entries. A matrix $M$ is called \emph{$\tau$-Hermitian} if it has entries in $L$ and is equal to its conjugate transpose. Observe that the matrix $J^{\vec\alpha}$ is diagonal with entries in $F$, and so $J^{\vec\alpha}$ is $\tau$-Hermitian. Furthermore, it is a standard result (see \cite[\S 6.8]{WitteMorris}, for example) that $\SU(J^{\vec\alpha},\O_L,\tau):=\{A\in \SL(n+1,\O_L)\mid A^\ast J^{\vec \alpha}A=J^{\vec\alpha}\}$ is an arithmetic lattice in $\SL(n+1,\R)$ that is cocompact if and only if $F\neq\Q$.

\section{The construction}\label{construction}

In this section we describe the the construction of the thin groups in Theorem \ref{mainthm}. Recall that $F$ is a totally real number field, $\alpha_1,\ldots,\alpha_n$ are positive elements of $F$ such that $s(\alpha_i)=0$.

Next, we construct a certain real quadratic extension of $L$. In order to proceed with the construction, we require the following:

\begin{lemma}\label{specialunits}
 Let $F$ be any totally real field and $N>0$, then $F$ contains infinitely many units $u$ with the properties that:

\begin{enumerate}
	\item  At the identity embedding of $F$, $u > N$
	\item At all the other
embeddings $\sigma : F\rightarrow \R$ one has $ 0 < \sigma(u) < 1$.\end{enumerate}
\end{lemma}
\begin{proof} Suppose that $[F:\Q] =k+1$ and let $v_1,....., v_k$ be generators of the 
unit group, \red{$\mathcal{O}_F^\times$}, as determined by Dirichlet's Unit Theorem. 

There is an embedding \red{$\sigma:F\to \R^{k+1}$ given by $\sigma(x)=(\sigma_1(x),\ldots,\sigma_{k+1}(x))$, where the $\sigma_i$ are all the embeddings of $F$ into $\R$, chosen so that $\sigma_1$ is the identity.} By replacing each $v_i$ with its square we can suppose that \red{$\sigma(v_i)$ is contained in the positive orthant of $\R^{k+1}$.} 
This will replace $\O_F^\times$ with a subgroup of finite index in $\O_F^\times$.

Taking \red{componentwise} logarithms gives a map, \red{$\log: \R^{k+1}_+\to \R^{k+1}$, where $\R^{k+1}_+$ is the positive orthant in $\R^{k+1}$. Furthermore, since each $v_i$ is a unit, it follows that $\log(\sigma(v_i))$} lies in the hyperplane where the sum of the coordinates is
equal to zero. Dirichlet's Unit Theorem implies that the
set \red{$B=\{\log(\sigma(v_1)),....., \log(\sigma(v_k))\}$} is a basis for this hyperplane, so there is
a linear combination of their images which yield the vector
$\vec{a}=(1, -1/k,-1/k,......, -1/k)$, \red{with respect to the basis $B$}, hence there is a rational linear combination
giving a vector very close to $
\vec{a}$. By scaling \red{to clear denominators},
one obtains an {\em integer} linear combination with the property that
the last $k$ coordinates are negative and the first coordinate is
positive. After possibly taking further powers (to arrange $u >N$) and exponentiating
one obtains a unit with the required properties. \end{proof}
\begin{remark}
	Notice that once a unit $u$ satisfies the above conditions, so
do all its powers.
\end{remark} 

Next, let $u$ be one of the units guaranteed by Lemma \ref{specialunits} for $N>2$. Note that by construction, $u^2-4>0$ \red{and $\sigma(u^2-4)=\sigma(u)^2-4<0$ for all non-identity embeddings of $F$. In particular, this implies that $u$ is not a square.} Let $s$ be a root of the polynomial $p_u(x)=x^2-ux+1$ and let $L=F(s)$. \red{Note that $u^2-4$ is the discriminant of this polynomial. By construction, $L$ is a real quadratic extension of $F$. Furthermore, since $\sigma(u^2-4)<0$ for all non-identity embeddings $L$ has exactly 2 real places.} Let $\tau:L\to L$ be the unique non-trivial Galois automorphism of $L$ over $F$. By construction, $s\in \O_L$ and since $\tau(s)$ is the other root of $p_u(x)$, a simple computation shows that $\tau(s)=1/s$, and so $s\in \O_L^\times$. With this in mind, we henceforth call elements $u\in L$ such that $\tau(u)=1/u$ \emph{$\tau$-unitary} or just \emph{unitary} if $\tau$ is clear from context. Note, that $\tau$-unitary elements in $\O_L$ are all units.

Every power of $s$ (and indeed $-s$) is also unitary. Furthermore, we note that these are the only possible unitary elements of $\O_L^\times$. The reason is this: notice that the rank of the unit group of $\O_F$ is $[F: \Q] - 1$. Also, $F(s)$ has two real embeddings, (coming from $s$ and $1/s$) and all the other embeddings lie on the unit circle (in other words, $s$ is a so-called {\em Salem number}) since we required the other embeddings of $u$ were less than $2$ in absolute value. So by Dirichlet's theorem, the unit group of  $\O_L$ has rank 
$$2 + ( 2[F: \Q] - 2)/2 -1 = [F: \Q],$$
which is 1 larger than the rank of $\O_F^\times$. Since $\tau$ induces an automorphism of the unit group that fixes $\O_F^\times$, the possibilities for are all accounted for by $s$ and its powers.

From the discussion of the previous section we can find torsion-free subgroups $\Gamma$ (resp.\ $\Gamma_1$) commensurable with $\SO(J^{\vec\alpha},\O_F)$ (resp.\ $\SO(J^{\vec\alpha_i},\O_F)$) such that $M_1:=\HH^{n-1}/\Gamma_1$ is an embedded submanifold of $M:=\HH^n/\Gamma$.  As previously mentioned, we can regard $(M,id)$ as a complete hyperbolic (and hence convex projective) structure on $M$ whose holonomy $\rho$ is the inclusion of $\Gamma$ into $\SL(n+1,\R)$. Since $M$ contains an embedded totally geodesic hypersurface, $M_1$, it is possible to bend $M$ along $M_1$ to produce a family of representations $\rho_t:\Gamma\to \SL(n+1,\R)$. We now show that for various special values of the parameter $t$, the group $\rho_t(\Gamma)$ will be a thin group inside a lattice in $\SL(n+1,\R)$. These special values turn out to be logarithms of unitary elements of $\O_L$.

The main goal of the remainder of this section is to prove the following theorem

\begin{theorem}\label{rhotinlambda}
	If $u\in\O_L$ is unitary and $t=\log\abs{u}$ then $\rho_t(\Gamma)\subset \SU(J^{\vec\alpha},\O_L,\tau)$. 
\end{theorem}

In order to prove Theorem \ref{rhotinlambda} we need a preliminary lemma. Recall that in Section \ref{bending} we defined for each $t\in \R$ the matrix
$$B_t=\begin{pmatrix}
	e^{-nt} & \\
	 & e^{t}I_{n}
\end{pmatrix}$$

\begin{lemma}\label{Bisunitary}
	If $u\in\O_L$ is unitary and $t=\log \abs{u}$
	\begin{itemize}
	\item 	$B_t\in \SU(J^{\vec\alpha},\O_L,\tau)$. 
	\item $B_t$ centralizes $\Gamma_1$. 
	\end{itemize}
 \end{lemma}
\begin{proof}

If $u\in \O_L$ is unitary then so is $-u$, and so without loss of generality we assume that $u>0$. Since $u$ is unitary we have
	$$B_t^\ast J^{\vec\alpha} B_t=\begin{pmatrix}
	u^{-n} & 0\\
	0 & u I_{n}	
	\end{pmatrix}
	\begin{pmatrix}
		\alpha_1 & 0\\
		0 & J^{\vec\alpha_1}
	\end{pmatrix}
	\begin{pmatrix}
		u^n & 0\\
		0 & u^{-1}I_n
	\end{pmatrix}=\begin{pmatrix}
		\alpha_1 & 0\\
		0 & J^{\vec\alpha_1}
	\end{pmatrix}=J^{\vec\alpha},$$
which proves that $B_t\in \SU(J^{\vec\alpha},\O_L,\tau).$

For the second point, let $\{e_1,\ldots,e_{n+1}\}$ be the standard basis for $\R^{n+1}$ end let $\{e_1^\ast,\ldots,e_{n+1}^\ast\}$ be the corresponding dual basis. For each $t$, $B_t$ acts trivially on the projective spaces corresponding to $\langle e_1\rangle$ and $\ker(e_1^\ast)$. By construction $\Gamma_1$ preserves both of these subspaces, and so $B_t$ centralizes $\Gamma_1$. 
\end{proof}

\begin{proof}[Proof of Theorem \ref{rhotinlambda}.]
	First, observe that $\Gamma\subset \SO(J^{\vec\alpha},\O_F)\subset \SU(J^{\vec\alpha},\O_L,\tau)$ for any $L=F(s)$. There are now two cases. If $M\bs M_1$ is separating then as describe in Section \ref{bending} $\Gamma$ splits as an amalgamated product $G_1\ast_{\Gamma_1}G_2$, and $\rho_t$ is defined by the property that $\rho_t(\gamma)=\rho_0(\gamma)$ if $\gamma\in G_1$ and $\rho_t(\gamma)=B_t\rho_0(\gamma)B_t^{-1}$ if $\gamma\in G_2$. By the previous observation $\rho_0(\gamma)\in \SU(J^{\vec\alpha},\O_L,\tau)$ for any $\gamma\in \Gamma$ and by Lemma \ref{Bisunitary} $B_t\in \SU(J^{\vec\alpha},\O_L,\tau)$. It follows that $\rho_t(\Gamma)\leq \SU(J^{\vec\alpha},\O_L,\tau)$.
	
The separating case is similar. In this case, $\Gamma=\Gamma'\ast_{s}$ is an HNN extension where $\Gamma'=\pi_1(M\bs M_{1})$ and $\rho_t$ is defined by the property that $\rho_t(\gamma)=\rho_0(\gamma)$ if $\gamma\in \Gamma'$ and $\rho_t(s)=B_t\rho_0(s)$. Using a similar argument as before it follows that $\rho_t(\Gamma)\leq \SU(J^{\vec\alpha},\O_L,\tau)$.  
	\end{proof}

\section{Certifying thinness}\label{thinness}

The goal of this section is to certify the thinness of the examples produced in the previous section. Before proceeding we recall some notation. $\Gamma$ and $\Gamma_1$ are finite index subgroups of $\SO(J^{\vec\alpha},\O_F)$ and $\SO(J^{\vec\alpha_1},\O_F)$ such that $M=\HH^n/\Gamma$ is a manifold and $M_1=\HH^{n-1}/\Gamma_1$ is an embedded totally geodesic submanifold. Furthermore, if $M$ is non-compact then all of the cusps are torus cusps and the intersection of $M_1$ with one of these cusps is connected. Let $\rho_t$ be obtained by bending $M$ along $M_1$, let $\Gamma_t=\rho_t(\Gamma)$. By Theorems \ref{bending} and \ref{balmarbending} there is a properly convex set $\Omega_t$ such that $M_t:=\Omega_t/\Gamma_t$ is a properly convex manifold that is diffeomorphic to $M$. Furthermore, if $M$ is non-compact then $M_t$ has generalized cusp ends  

The main theorem is a corollary of the following result.
\begin{proposition}\label{threesteps} Suppose that $\rho_t$ is obtained by bending $M$ along $M_1$ then
	\begin{enumerate}
	\item For every $t$, $\rho_t$ is injective,
	\item If $u \in\O_L$ is unitary and $t=\log\abs{u}$ then $\rho_t(\Gamma)$ has infinite index in $\SU(J^{\vec\alpha},\O_L,\tau)$, and 
	\item For any $t\neq 0$, $\rho_t(\Gamma)$ is Zariski dense in $\SL(n+1,\R)$ \end{enumerate}
	In particular, $\SU(J^{\vec\alpha},\O_L,\tau)$ contains a thin group isomorphic to $\pi_1M$. 
\end{proposition}
\begin{proof}
	The first two points are simple. For (1) observe that by Theorem \ref{bending}, $\rho_t$ is the holonomy of a convex projective structure on $M$. 
	
	Let $\Gamma_t=\rho_t(\Gamma)$. For (2), we can use the fact that the manifold $\HH^n/\Gamma$ contains an embedded hypersurface, as
	we observed earlier. It follows from \cite{MillsonBetti} that the group $\Gamma$ virtually surjects \red{onto} $\Z$ \red{and thus has infinite abelianization}. \red{More precisely, one can pass to a finite cover, $M'$, of $\HH^n/\Gamma$ that contains an embedded non-separating hypersurface, $\Sigma'$. There is a non-trivial cohomology class in $H^1(M',\Z)$ that is Poincar\'e dual to $\Sigma'$, which gives the virtual surjection onto $\Z$.} Since  $\SU(J^{\vec\alpha},\O_L,\tau)$ is a lattice in a high rank Lie group, it follows that it has property (T) (see \cite[Prop.\ 13.4.1]{WitteMorris}). Furthermore, any finite index subgroup  of $\SU(J^{\vec{\alpha}},\O_L,\tau)$ will also have  property (T) and thus will have finite abelianization (see \cite[Cor.\ 13.1.5]{WitteMorris}). Since the groups $\Gamma_t$ are all abstractly isomorphic it follows that $\Gamma_t$ is not a lattice, this implies $(2)$.

	The third point breaks into two cases depending on whether or not $\Gamma$ is a cocompact lattice in $\SO(n,1)$. We treat the cocompact case first. By Theorem \ref{Marquisbending}, it follows that $\Gamma_t$ acts cocompactly on a properly convex set $\Omega_t$. Since $\Gamma$ is a cocompact lattice in $\SO(n,1)$, the group $\Gamma$ is word hyperbolic and it follows from work of Benoist \cite{BenoistCDI} that for each $t$ the domain $\Omega_t$ is strictly convex. Hence $\Omega_t$ cannot be written as a non-trivial product of properly convex sets. Applying \cite[Thm 1.1]{BenoistCDI} it follows that $\Gamma_t$ is either Zariski dense or $\Omega_t$ is the projectivization of an irreducible symmetric convex cone. Suppose we are in the latter case. Irreducible symmetric convex cones were classified by Koecher (see \cite[Fact 1.3]{BenoistCDIII} for a precise statement) and since $\Omega_t$ is strictly convex it follows that $\Omega_t\cong \HH^n$. It follows that $\Gamma_t$ is conjugate to a lattice in $\SO(n,1)$, which by Mostow rigidity must be $\Gamma$. However, bending in this context never produces conjugate representations, since
	any such conjugacy would centralize the subgroup corresponding to the  complement of the bending hypersurface. However this subgroup is nonelementary
	and this is a contradiction. Therefore, $\Gamma_t$ is Zariski dense if $t\neq 0$, which concludes the cocompact case. 
	
	The non-cocompact case is an immediate corollary of the following Proposition whose proof occupies the remainder of this section.
\end{proof}

\begin{proposition}\label{noncptZariskidense}
	If $M$ is non-compact, $\rho_t$ is obtained by bending $M$ along $M_1$, and $\Gamma_t=\rho_t(\Gamma)$ then $\Gamma_t$ is Zariski dense. \end{proposition}

The strategy for proving Proposition \ref{noncptZariskidense} is to apply the following two results from \cite{BenoistCones}.

\begin{theorem}[Lem.\ 3.9 of \cite{BenoistCones}]\label{trans}
	Suppose that $G\subset \SL(V)$ is a connected, semisimple, proximal Lie subgroup acting irreducibly on $V$. If $G$ acts transitively on $P(V)$ then either $V=\R^{n}$ and $G=\SL(n,\R)$ or $V=\R^{2n}$ and $G={\rm Sp}(2n,\R)$. 
\end{theorem}

The next Theorem allows us to rule out the second possibility in our case of interest. 

\begin{theorem}[Cor.\ 3.5 of \cite{BenoistCones}]\label{nosmpylectic}
	If $\Gamma\subset \SL(V)$ acts strongly irreducibly on $V$ and preserves an open properly convex subset then $\Gamma$ does not preserve a symplectic form. 
\end{theorem}

\begin{proof}[Proof of Proposition \ref{noncptZariskidense}.]
	Let $G_t$ be the Zariski closure of $\Gamma_t$ and let $G^0_t$ be the identity component of $G_t$. We now show that $G^0_t=\SL(n+1,\R)$. By applying Lemma \ref{limitset} we see that $G^0_t$ satisfies all of the hypotheses of Theorem \ref{trans} except for transitivity. 
	
	Since the intersection of $M_1$ with one of the cusps of $M$ is connected we can apply \cite[Thm.\ 6.1]{BalMar} to conclude that $M_t$ has at least one type 1 cusp. It follows that (after possibly conjugating) $G^0_t$ contains the Zariski closure of $P_1$. Since $\Gamma_t$ acts irreducibly on $V$ it is not the case that $\Lambda_{G^0_t}$ is contained in $\ker(e_2^\ast)\cup \ker(e_{n+1}^\ast)$, and \red{so by Lemma \ref{p1barorbits}} we can choose a point $x\in \Lambda_{G^0_t}$ such that $\overline{P_1}\cdot x$ is open in $P(V)$. It follows that $G^0_t\cdot x$ has non-empty interior and is hence open. Finally, by Lemma \ref{limitset}, $G^0_t\cdot x=\Lambda_{G^0_t}$, which is closed, hence $G^0_t$ acts transitively on $P(V)$. 
	
	Finally, by Theorem \ref{nosmpylectic}, $\Gamma_t$ does not preserve a symplectic form and hence neither does $G^0_t$. Applying Theorem \ref{trans} it follows that $G^0_t=\SL(n+1,\R)$.   
\end{proof}

We can now prove the Theorem \ref{mainthm}.

\begin{proof}[Proof of Theorem \ref{mainthm}.]
	Since $\Gamma$ is an arithmetic group of orthogonal type in $\SO(n,1)$ there is a totally real number field $F$ with ring of integers $\O_F$ as well as $\vec\alpha=(\alpha_1,\ldots,\alpha_n)$ such that $\Gamma$ is commensurable with $\SO(J^{\vec\alpha},\O_F)$. The group $\Gamma$ is cocompact if and only if $F\neq \Q$.  
	
	By using standard separability arguments, we can pass to a finite index subgroup $\Gamma'$ such that $M=\HH^n/\Gamma'$ contains an embedded totally geodesic hypersurface $M_1$ with the property that if $M$ is non-compact it has only torus cusps and such that $M_1$ has connected intersection with at least one of the cusps. 
	
	Let $\rho_t$ be obtained by bending $M$ along $M_1$. Let $v\in \O_F^\times$ be an element guaranteed by Lemma \ref{specialunits} and let $L=F(s)$, where $s$ is a root of $p_v(x)$, and let $\tau$ be the non-trivial Galois automorphism of $L$ over $F$. Next, let $u=s^n$ be a  $\tau$-unit in $\O_F^\times$.  If $t=\log\abs{u}$ then by Theorem \ref{rhotinlambda} it follows that $\rho_t(\Gamma')\subset \SU(J^{\vec\alpha},\O_L,\tau)$. Furthermore, by Theorem \ref{threesteps}, $\rho_t(\Gamma')$ is a thin subgroup of $\SU(J^{\vec\alpha},\O_L,\tau)$. Again, $\SU(J^{\vec\alpha},\O_L,\tau)$ is cocompact if and only if $F\neq \Q$ and by varying $v$ and $\vec{\alpha}$ it is possible to produce infinitely many non-commensurable lattices.   
\end{proof}

\bibliographystyle{plain}
\bibliography{bibliography}

\end{document}